\documentstyle[12pt]{article}

\begin{document}

\newtheorem{theorem}{Theorem}[section]
\newtheorem{corollary}[theorem]{Corollary}
\newtheorem{definition}[theorem]{Definition}
\newtheorem{conjecture}[theorem]{Conjecture}
\newtheorem{question}[theorem]{Question}
\newtheorem{lemma}[theorem]{Lemma}
\newtheorem{proposition}[theorem]{Proposition}
\newtheorem{example}[theorem]{Example}
\newtheorem{problem}[theorem]{Problem}
\newenvironment{proof}{\noindent {\bf
Proof.}}{\rule{3mm}{3mm}\par\medskip}
\newcommand{\remark}{\medskip\par\noindent {\bf Remark.~~}}
\newcommand{\pp}{{\it p.}}
\newcommand{\de}{\em}

\newcommand{\JEC}{{\it Europ. J. Combinatorics},  }
\newcommand{\JCTB}{{\it J. Combin. Theory Ser. B.}, }
\newcommand{\JCT}{{\it J. Combin. Theory}, }
\newcommand{\JGT}{{\it J. Graph Theory}, }
\newcommand{\ComHung}{{\it Combinatorica}, }
\newcommand{\DM}{{\it Discrete Math.}, }
\newcommand{\ARS}{{\it Ars Combin.}, }
\newcommand{\SIAMDM}{{\it SIAM J. Discrete Math.}, }
\newcommand{\SIAMADM}{{\it SIAM J. Algebraic Discrete Methods}, }
\newcommand{\SIAMC}{{\it SIAM J. Comput.}, }
\newcommand{\ConAMS}{{\it Contemp. Math. AMS}, }
\newcommand{\TransAMS}{{\it Trans. Amer. Math. Soc.}, }
\newcommand{\AnDM}{{\it Ann. Discrete Math.}, }
\newcommand{\NBS}{{\it J. Res. Nat. Bur. Standards} {\rm B}, }
\newcommand{\ConNum}{{\it Congr. Numer.}, }
\newcommand{\CJM}{{\it Canad. J. Math.}, }
\newcommand{\JLMS}{{\it J. London Math. Soc.}, }
\newcommand{\PLMS}{{\it Proc. London Math. Soc.}, }
\newcommand{\PAMS}{{\it Proc. Amer. Math. Soc.}, }
\newcommand{\JCMCC}{{\it J. Combin. Math. Combin. Comput.}, }
\newcommand{\GC}{{\it Graphs Combin.}, }
\thispagestyle{empty}
\title{The  Wiener Index of Unicyclic Graphs with Girth and the Matching Number\thanks{
 Supported by  Key Project of  Lishui University (No. KZ201015), National Natural Science Foundation of China
(No.10971137), National Basic Research Program of China 973 Program
(No.2006CB805900) and a grant of Science and Technology Commission
of Shanghai Municipality (STCSM No: 09XD1402500).\newline \indent
$^{\dagger}$Correspondent author: Xiao-Dong Zhang (Email:
xiaodong@sjtu.edu.cn)}}

\author{ Ya-Hong Chen\\
{\small Teacher  Education  College,
 Lishui University}\\
{\small Lishui, Zhejiang 323000, PR China}\\
  Xiao-Dong Zhang$^{\dagger}$,  \\
{\small Department of Mathematics,
 Shanghai Jiao Tong University} \\
{\small  800 Dongchuan road, Shanghai, 200240, P.R. China}\\
}
\date{}
\maketitle
 \thispagestyle{empty}

 \begin{minipage}{5in}
 \begin{center}
 Abstract
 \end{center}

   In this paper, we investigate how the Wiener index of unicyclic graphs varies with
    graph operations. These results are used to present a sharp lower
    bound for the Wiener index of unicyclic graphs of order $n$ with
    girth and the matching number $\beta\ge \frac{3g}{2}$. Moreover, we  characterize all
    extremal graphs which attain the lower bound.
   \end{minipage}
 \vskip 0.5cm
 {{\bf Key words:} Wiener index, unicyclic, the matching number,
girth.
 \vskip 0.3cm
      {{\bf MSC2010:} 05C12, 05C70.}
      }
\vskip 0.5cm
\section{Introduction}

{\it The Wiener index} \cite{wiener1947} of a simple connected
graph is the sum of distances between all pairs of vertices, which
has been much studied in both mathematical (see
\cite{dobrynin2001,entringer1976, fischermann2002, fischermann2005,
Gutman1997-k, gutman-k-2}) and chemical (see \cite{gutman2000,
gutman1997-p,gutman1997,hosoya1971,klein1997,lukovits1991,polansky1986,
rouvray2002}) literatures. Gutman et al. in \cite{gutman1997-p} gave
some results for the Wiener indices of a unicyclic graph, which is a
connected graph with a unique cycle.
 Recently, Yan and Yeh
\cite{yan2006} investigated the relations between the matching
number and the Wiener index.  Du and Zhou in \cite{du2009}
determined the minimum Wiener index in the set of unicyclic graphs
of order $n$ with girth and the number of pendent vertices.
Moreover,  Du and Zhou in \cite{du2010} gave the sharp lower bounds
for the Wiener index of unicyclic graphs with the matching number.
On the Wiener index and related problems of trees and unicyclic
graphs may be referred to \cite{du2009-1, zhang2008, zhang2010}.

  Through this paper, all graphs are finite, simple and undirected.
    Let $G= (V,~E)$ be a simple connected graph with
vertex set $V(G)=\{v_1,\cdots, v_n\}$ and edge set $E(G)$. The {\it
girth } of a graph $G$ with a cycle is the  length of its shortest
cycle. A {\it matching} in a graph $G$ is a set of edges with no
shared end vertices. The {\it matching number } of a graph $G$ is
the maximum size of all matching of graphs, and denoted by
$\beta(G)$ or $\beta$.
The {\it distance} between vertices $v_i$ and $v_j$ is the
minimum number of edges between $v_i$ and $v_j$ and denoted by
$d_G(v_i, v_j)$ (or for short $d(v_i,v_j)$). The {\it Wiener index}
of a connected graph $G$ is defined as
\begin{equation}\label{weiner-def}
W(G)=\sum_{\{v_i, v_j\}\subseteq V(G)}d(v_i, v_j).
\end{equation}
Moreover, the {\it distance of a vertex} $v$, denoted by $d_G(v)$,
is the sum of of distances between $v$ and all other vertices of
$G$. Then
\begin{equation}\label{wiener-ano}
W(G)=\frac{1}{2}\sum_{u\in V(G)}d_G(u).
\end{equation}
 In this paper, motivated by the above results, we investigate, in Section 2,  how the
 Wiener index of unicyclic graphs with  girth and the matching
 number varies with some graph operations. In Section 3,  we obtain
 a sharp lower bound for the Wiener index in the set consisting of
 all unicyclic graphs of order $n$ with girth $g$ and the matching
 number $\beta\ge\frac{3g}{2}$. Moreover, the all extremal graphs which attain the
 lower bound have been characterized.

\section{Wiener index  with switching operations }

Let $G=(V(G), E(G)$ be a unicyclic graph of order $n$ with girth
$g$. Suppose that the only cycle is $C_g=u_1u_2\cdots u_g$. Then
$G-E(C_g)$, which is obtained from $G$ by deleted all edges of
$C_g$, has $g$ connected components, each of which is  tree $T_i$ of
order $n_i$ for $i=1, \cdots, g$. The connected component $T_i$ is
called a {\it branch} of $G$ at $u_i$.  Clearly,
$n_1+n_2+\cdots+n_g=n$. Moreover, any unicyclic graph of order $n$
with girth $g$ can be denoted by $U(T_1, \cdots, T_g)$.
 Let ${\mathcal{U}}_{(n, g, \beta)}$ be the set of all unicyclic
graphs of order $n$ with girth $g$ and the matching number $\beta$.
It is easy to get the following result.
\begin{lemma}\label{lemma2-1}
Let $G=U(T_1, \cdots, T_g)$ be a unicyclic graph of order $n$ with
girth $g$. Then
\begin{eqnarray}\label{lemma2-1-eq}
&W(G)&= (n-\frac{g}{2})\lfloor
\frac{g^2}{4}\rfloor+(g-1)\sum_{i=1}^gd_{T_i}(u_i)+\sum_{i=1}^gW(T_i)+\\
&& \sum_{i=1}^{g-1}\sum_{j=i+1}^g
[(n_i-1)d_{T_j}(u_j)+(n_j-1)d_{T_i}(u_i)+(n_i-1)(n_j-1)d_{C_g}(u_i,u_j)].\nonumber
\end{eqnarray}
\end{lemma}
\begin{proof} By the definition of the Wiener and
$W(C_g)=\frac{g}{2}\lfloor\frac{g^2}{4}\rfloor,$ we have
\begin{eqnarray}
W(G)&=&\sum_{i=1}^g\sum_{\{u, v\}\subseteq V(T_i)}d_{T_i}(u,
v)+\sum_{i=1}^{g-1}\sum_{j=i+1}^g\sum_{u\in V(T_i), v\in
V(T_j)}d_G(u,v)\nonumber\\
&=&\sum_{i=1}^gW(T_i)+\sum_{i=1}^{g-1}\sum_{j=i+1}^g\sum_{u\in
V(T_i), v\in V(T_j)}[d_{T_i}(u, u_i)+d_{C_g}(u_i, u_j)+d_{T_j}(u_j,
v)]
\nonumber\\
&=&\sum_{i=1}^gW(T_i)+\sum_{i=1}^{g-1}\sum_{j=i+1}^g[n_jd_{T_i}(u_i)+n_in_jd_{C_g}(u_i,
u_j)+n_id_{T_j}(u_j)]\\
&=& (n-\frac{g}{2})\lfloor
\frac{g^2}{4}\rfloor+(g-1)\sum_{i=1}^gd_{T_i}(u_i)+\sum_{i=1}^gW(T_i)+ \nonumber\\
&& \sum_{i=1}^{g-1}\sum_{j=i+1}^g
[(n_i-1)d_{T_j}(u_j)+(n_j-1)d_{T_i}(u_i)+(n_i-1)(n_j-1)d_{C_g}(u_i,u_j)].
\nonumber\\ \nonumber
\end{eqnarray}
Hence  (\ref{lemma2-1-eq}) holds.
\end{proof}

\begin{corollary}\label{cor2-2}
Let $G=U(T_1, \cdots, T_g)$ and $G_1=U(\widetilde{T_1},T_2, \cdots,
T_g)$ be two unicyclic graphs of order $n$ with girth $g$. If
$|V(T_1)=|V( \widetilde{T_1})|=n_1, W(T_1)\ge W(\widetilde{T_1}) $
and $d_{T_1}(u_1)\ge d_{\widetilde{T_1}}(u_1)$, then
\begin{equation}\label{cor2-2-eq}
W(G)\ge W(G_1)
\end{equation}
with equality if and only if  $W(T_1)= W(\widetilde{T_1}) $ and
$d_{T_1}(u_1)= d_{\widetilde{T_1}}(u_1)$.
\end{corollary}
\begin{proof}
By (\ref{lemma2-1-eq}) in Lemma~\ref{lemma2-1}, we have
\begin{eqnarray*}
W(G)-W(G_1)&=&
W(T_1)-W(\widetilde{T_1})+(g-1)(d_{T_1}(u_1)-d_{\widetilde{T_1}}(u_1))
+ \\
 &&\sum_{j=2}^g (n_j-1)(d_{T_1}(u_1)-d_{\widetilde{T_1}}(u_1))\ge
 0,
\end{eqnarray*}
since $W(T_1)\ge W(\widetilde{T_1}) $ and $d_{T_1}(u_1)\ge
d_{\widetilde{T_1}}(u_1)$. Hence the assertion holds.
\end{proof}

For given two nonnegative integers $a, b$, let $T_{a, b}^* $ be a
rooted tree of order $2a+b+1$ obtained from
 the root star $K_{1,a+b}$ at root $u_1$ by adding
$a$ pendent edges to $a$  pendent vertices of $K_{1,a+b}$. In
particular, $T_{0,0}^*$ is an isolated vertex.  Then the matching
number of $T_{a, b}^*$ is $a+1$ for $b\ge 1$ and $a$ for $b=0$.


\begin{lemma}
\label{lem-2-3} Let $G=U(T_1, \cdots, T_g)$ be a unicyclic graph of
order $n$ with  girth $g$ and the matching number $\beta$. Denote by
$\beta_1$ the matching number of $T_1$ of order $n_1$. If
$\beta_1=0$ or  $G$ has a maximum matching $M$ containing an edge
$u_1x$, $x\in V(T_1)$, let $G_1 $ be the unicyclic graph from $G$ by
replacing  $T_1$ with  $T_{0,0}^* $ for $\beta_1=0 $ and replacing
$T_1$ with $T_{\beta_1-1, n_1-2\beta_1+1}^*$.
Then
the matching number of $G_1$ is $\beta$ and
\begin{equation}\label{lem2-3-eq}
W(G)\ge W(G_1)
\end{equation}
with equality if and only if  $T_1=T_{0,0}^*$ or $T_{\beta_1-1,
n_1-2\beta_1+1}^*$.
\end{lemma}
\begin{proof} If $\beta_1=0$, the assertion obviously holds. Assume
that $\beta_1\ge 1$. Moreover, since $G$ has a maximum matching $M$
containing an edge $u_1x$, $x\in V(T_1)$, it is easy to see that the
matching number of $G_1$ is $\beta$.  Since the matching number of
$T_1$ is $\beta_1$, there exist at most $n_1-\beta_1$ vertices
adjacent to $u_1$ in $T_1$ (otherwise the matching number of $T_1$
is less than $\beta_1$). Hence $d_{T_1}(u_1)\ge
(n_1-\beta_1)+2(n_1-(n_1-\beta_1+1))=n_1+\beta_1-2$. Further by
Corollary~5.7 in \cite{zhang2008}, we have $W(T_1)\ge
W(T_{\beta_1-1, n_1-2\beta_1+1}^*)$ with equality if and only if
$T_1$ is $T_{\beta_1-1, n_1-2\beta_1+1}^*$. Hence by
Corollary~\ref{cor2-2}, the assertion holds.
\end{proof}

\begin{lemma}\label{lem-2-4}
Let $T$ be a tree of order $n\ge 3$ and $u\in V(T)$. Suppose that
the matching number of $T-u$ is $\beta$ and $T-u$ has $p$ connected
components $T_1, \cdots, T_p$ of order $n_1, \cdots, n_p$,
respectively. Then
\begin{equation}
W(T)\ge W(T_{\beta, n-2\beta-1}^*)=n^2+(\beta-2)n+(-3\beta+1)
\end{equation}
 with equality if and only if $T$ is $T_{\beta, n-2\beta-1}^*$.
  \end{lemma}
\begin{proof}
If $\beta=0$, then $T$ must be  the star graph $K_{1, n-1}$, which
is exact $T_{0, n-1}^*$. Hence without loss of generality, assume
that $0<\beta\le\frac{n-1}{2}$ and the matching number of $T_i$ is
$\beta_i\ge 1$ for $i=1, \cdots, t$ and 0 for $t+1\le i\le p$.
Assume the neighbor set of $u$ is $\{w_1, \cdots, w_p\}$. Then by
Theorem~4  in \cite{dobrynin2001} and Corollary~5.7 in
\cite{zhang2008}, we have
\begin{eqnarray*}
&&W(T)=n(n-1)+ \sum_{i=1}^p[W(T_i)+(n-n_i)d_{T_i}(w_i)-n_i^2]\\
&\ge
&n(n-1)+\sum_{i=1}^t[n_i^2+(\beta_i-3)n_i+(-3\beta_i+4)+(n-n_i)(n_i+\beta_i-2)-n_i^2]-(p-t)\\
&=&(n-1)^2+(n-3)\beta+(-2n+4)t+n(n-p+t-1)-\sum_{i=1}^tn_i^2\\
&\ge & (n-1)^2+(n-3)\beta+(-2n+4)t+n(n-p+t-1)\\
&& -4(t-1)-[(n-p+t-1)-2(t-1)]^2\\
&=& (n-1)^2+(n-3)\beta-2nt+8t-4t^2-p^2+(n-2t+2)p+(3t-3)(n+t-1)\\
&\ge&(n-1)^2+(n-3)\beta-2nt+8t-4t^2\\
&&-(n-t-1)^2+(n-2t+2)(n-t-1)+(3t-3)(n+t-1)\\
&=&n^2+(\beta-2)n+(-3\beta+1)=W(T_{\beta,n-2\beta-1}^*),
\end{eqnarray*}
where $n=n_1+\cdots+n_t+p-t+1\ge p+t+1$, and $d_{T_i}(w_i)\ge
n_i+\beta_i-2$, since $w_i$ is at most adjacent to $n_i-\beta_i$
vertices in $T_i$. Moreover, if  equality holds, then
$n_1=\cdots=n_t=2$, which implies $t=\beta$ and $p=n-\beta-1$.
Therefore $T$ must be $T_{\beta,n-2\beta-1}^*$.
\end{proof}

\begin{lemma}\label{lem2-5}
Let $G=U(T_1, \cdots, T_g)$ be a unicyclic graph of order $n$.
Suppose that any maximum matching of $G$ does not contain $u_1x$,
$x\in V(T_1)$ and the matching number of $T_1-u_1$  of order $n_1-1$
is $\beta_1$. Let $G_1=U(
T_{\beta_1, n_1-2\beta_1-1}^*, T_2, \cdots, T_g)$ be the unicyclic
graph obtained from $G$ by replacing $T_1$ with $T_{\beta_1,
n_1-2\beta_1-1}^*$.
 Then the matching
numbers of $G$ and $G_1$ are equal and
\begin{equation}\label{lem2-5-eq}
W(G)=W(U(T_1, \cdots, T_g))\ge W(G_1)=W(U(T_{\beta_1,
n_1-2\beta_1-1}^*, T_2, \cdots, T_g))
\end{equation}
with equality if and only if $T_1=T_{\beta_1, n_1-2\beta_1-1}$.
\end{lemma}
\begin{proof} It is easy to see that the matching number of $G$ is
equal to the matching number of $G_1$ by the definition, since any
maximum matching of $G$ does not contain $u_1x$, $x\in V(T_1)$. On
the other hand, by Lemma~\ref{lem-2-4}, we have $W(T_1)\ge
W(T_{\beta_1, n_1-2\beta_1-1}^*)$.
Let $T_1-u_1$ has $p$ components $T_{11},\cdots, T_{1p}$ with
$$|V(T_{11})|\ge \cdots \ge |V(T_{1t})|>
|V(T_{1,t+1})|=\cdots=|V(T_{1p})|=1.$$
 Since the matching number of
$T_1-u_1$ is $\beta_1$,  we have $t\le \beta_1\le \frac{n_1-1}{2}$
and $n_1\ge 2\beta_1+p-t+1\ge p+\beta_1+1$. Hence $d_{T_1}(u_1)\ge
p+2(n_1-p-1)=2n_1-p-2\ge n_1+\beta_1-1= d_{T_{\beta_1,
n_1-2\beta_1-1}^*}(u_1)$.
Therefore by Corollary~\ref{cor2-2}, we have $W(G)\ge W(G_1)$ with
equality if and only if $T_1=T_{\beta_1, n_1-2\beta_1-1}^*$.
\end{proof}

\begin{lemma}\label{lem2-6}
Let $G=U(T_1, \cdots, T_g)$ be a unicyclic graph of order $n$ with
girth $g$. Suppose that  $T_p$ of order $|V(T_p)|\ge 3$ and $T_q$ of
order $|V(T_q)|\ge 3 $ have pendant edges $u_px$ and $u_qy$,
respectively. Let $T_p^{(1)}$ be the tree from $T_p$ and $T_q$ by
identifying $u_p$ and $u_q$ and deleting the edge $u_qy$, and let
$T_q^{(1)}$ be the edge $u_qy$. Moreover, let $T_p^{(2)}$ be the
edge $u_px$, and let $T_q^{(2)}$ be the tree from $T_p$ and $T_q$ by
identifying $u_p$ and $u_q$ and deleting the edge $u_px$. Further,
let $G_i=U(T_1, \cdots, T_p^{(i)}, \cdots, T_q^{(i)}, \cdots, T_g)$
for $i=1,2$. Then the matching numbers of $G$, $G_1, G_2$ are equal,
and
\begin{equation}\label{lem2-6-eq}
W(G)> \min\{W(G_1), W(G_2)\}.
\end{equation}
\end{lemma}
\begin{proof}
Clearly, by the definition, the matching numbers of $G$, $G_1$ and
$G_2$ are equal. By  Lemma~\ref{lemma2-1}, it is easy to see that
\begin{equation}\label{lem2-6-eq2}
W(G)-W(G_1)=(n_p-2)(n_q-2)d_G(u_p, u_q)+(n_q-2)\sum_{i=1, i\neq p,
q}^gn_i[d_G(u_q, u_i)-d_G(u_p, u_i)].
\end{equation}
Similarly, we have
\begin{equation}\label{lem2-6-eq3}
W(G)-W(G_2)=(n_p-2)(n_q-2)d_G(u_p, u_q)-(n_p-2)\sum_{i=1, i\neq
p,q}^gn_i[d_G(u_q, u_i)-d_G(u_p, u_i)].
\end{equation}
Hence by (\ref{lem2-6-eq2}) and (\ref{lem2-6-eq3}), the assertion
holds.
\end{proof}
Similarly, we have the following result.
\begin{lemma}\label{lem2-7}
Let $G=U(T_1, \cdots, T_g)$ be a unicyclic graph of order $n$ with
girth $g$. Suppose that   $T_p$ of order $|V(T_p)|\ge 3$ has no
pendant vertices adjacent to $u_p$ and $T_q$ of order  $|V(T_q)|\ge
3 $ has an pendant adjacent vertex $y$ adjacent to $u_q$. Let
$T_p^{(1)}$ be the tree from $T_p$ and $T_q$ by identifying $u_p$
and $u_q$ with deleting the edge $u_qy$, and let $T_q^{(1)}$ be the
edge $u_qy$. Moreover, let $T_p^{(2)}$ be  isolated vertex $u_p$,
and let $T_q^{(2)}$ be the tree from $T_p$ and $T_q$ by identifying
$u_p$ and $u_q$. Further, let $G_i=U(T_1, \cdots, T_p^{(i)}, \cdots,
T_q^{(i)}, \cdots, T_g)$ for $i=1,2$. Then
\begin{equation}\label{lem2-6-eq}
W(G)> \min\{W(G_1), W(G_2)\}.
\end{equation}
\end{lemma}
\begin{proof}
By Lemma~\ref{lemma2-1} and some computation, it is easy to see that
$$W(G)-W(G_1)=(n_p-2)(n_q-2)d_G(u_p, u_q)+(n_q-2)\sum_{i=1, i\neq p,
q}^gn_i(d_G(u_q, u_i)-d_G(u_p, u_i)).$$ Similarly, we have
$$W(G)-W(G_2)=(n_p-1)(n_q-1)d_G(u_p, u_q)-(n_p-1)\sum_{i=1, i\neq p,
q}^gn_i(d_G(u_q, u_i)-d_G(u_p, u_i)).$$ Hence the assertion holds.
\end{proof}
\begin{corollary}\label{corollary2-8}
Let $G=U(T_{a_1, b_1}^*, \cdots, T_{a_g, b_g}^*)$ be a unicyclic
graph of order $n$ with girth $g$. If  $a_p\ge 1$, $b_p=0$, and
$b_q>0$, $2a_q+b_q\ge 2$ for $1\le p\neq q\le g$, then there exists
a unicyclic graph $G^{\prime}$ of order $n$ and girth $g$ such that
the matching numbers of $G$ and $G^{\prime}$ are equal and
$W(G)>W(G^{\prime})$.
\end{corollary}
\begin{proof}  Clearly, $|V(T_p)|\ge 3$ and $|V(T_q)|\ge
3$. Let $G_1=U(T_{a_1, b_1}^*,\cdots,    T_{a_p+a_q, b_q-1}^*,$ $
\cdots, T_{0, 1}^*,$ $ \cdots, T_{a_g, b_g}^*)$,\ \ $G_2=U(T_{a_1,
b_1}^*,T_{0,0}^*, \cdots, T_{a_p+a_q, b_q}^*, \cdots, T_{a_g,
b_g}^*).$ \ \   By Lemma~\ref{lem2-7}, we have $W(G)>\min\{W(G_1), \
W(G_2)\}.$ Moreover, let $\beta, \beta_1, \beta_2$ be the matching
numbers of  $G$, $G_1$ and $G_2$, respectively. Then
$\beta=\beta_2\le \beta_1\le \beta +1$. If $\beta_1=\beta+1$, let
$G_3=U(T_{a_1, b_1}^*,\cdots, T_{a_p+a_q-1, b_q+1}^*, \cdots, T_{0,
1}^*, \cdots, T_{a_g, b_g}^*)$. It is easy to see that
$W(G_1)>W(G_3)$ and the matching number of $G_3$ is $\beta$. Hence
the assertion holds.
\end{proof}

\begin{lemma} \label{lem2-8}
Let $G=U(T_1, \cdots, T_g)$ be a unicyclic graph of order $n$ with
girth $g$ and $|V(T_i)|=n_i$ for $i=1, \cdots, g$. Let $T_p^{(1)}$
be the tree from $T_p$ and $T_q$ by identifying $u_p$ and $u_q$, and
let $T_q^{(1)}$ be the isolated vertex. Moreover let $T_p^{(2)}$ be
the isolated vertex $u_p$, and let $T_q^{(2)}$ be the tree from
$T_p$ and $T_q$ by identifying $u_p$ and $u_q$. Further, let
$G_i=U(T_1, \cdots, T_p^{(i)}, \cdots, T_q^{(i)}, \cdots, T_g)$ for
$i=1,2$. Then
\begin{equation}\label{lem2-6-eq}
W(G)> \min\{W(G_1), W(G_2)\}.
\end{equation}
\end{lemma}
\begin{proof}  Assume that $p< q$. By Lemma~\ref{lemma2-1}
and some computation, it is easy to see that
$$W(G)-W(G_1)=(n_p-1)(n_q-1)d_G(u_p, u_q)+(n_q-1)\sum_{i=1, i\neq p,
q}^gn_i(d_G(u_q, u_i)-d_G(u_p, u_i))$$
and
$$W(G)-W(G_2)=(n_p-1)(n_q-1)d_G(u_p, u_q)-(n_p-1)\sum_{i=1, i\neq p,
q}^gn_i(d_G(u_q, u_i)-d_G(u_p, u_i)).$$ Hence it is easy to see that
the assertion holds.
\end{proof}

\begin{corollary}\label{corollary2-10}
Let $G=U(T_{a_1, b_1}^*, \cdots, T_{a_p, 0}^*, \cdots, T_{a_q,
0}^*,\cdots, T_{a_g,b_g}^*)$ be a unicyclic graph of order $n$ with
girth $g$. If $a_p, a_q\ge 1$, let $G_1=U(T_{a_1, b_1}^*, \cdots,
T_{a_p+a_q, 0}^*, \cdots, T_{0, 0}^*,\cdots, T_{a_g,b_g}^*)$ and
$G_2=U(T_{a_1, b_1}^*, \cdots, T_{0, 0}^*, \cdots, T_{a_p+a_q,
0}^*,\cdots, T_{a_g,b_g}^*)$, then  the matching numbers of $G$,
$G_1$ and $G_2$ are equal and $ W(G)> \min\{W(G_1), W(G_2)\}. $
\end{corollary}
\begin{proof}
It follows from Lemma~\ref{lem2-8} that the assertion holds.
\end{proof}
 Now we can present the main result in this section.
\begin{theorem}\label{th2-8}
Let $G=U(T_1, \cdots, T_g)$ be a unicyclic graph of order $n$ with
girth $g$. Then there exist  nonnegative integers $a_1, b_1, \cdots,
b_g$ with $b_j\le 1$ for $j=2, \cdots, g$ such that
such that $G$ and $\widetilde{G}=U(T_{a_1, b_1}^*, T_{0, b_2}^*,
\cdots, T_{0,b_g}^*)$ have the same the matching number and
\begin{equation}\label{theorem2-8-eq}
W(G)=W(U(T_1, \cdots, T_g))\ge W(\widetilde{G})=W(U(T_{a_1, b_1}^*,
T_{0, b_2}^*, \cdots, T_{0,b_g}^*))
\end{equation}
with equality if and only if $G=U(T_{a_1, b_1}^*, T_{0, b_2}^*,
\cdots, T_{0,b_g}^*)$.
\end{theorem}
\begin{proof}  We consider the following two cases.

{\bf Case 1:}  $|V(T_i)|\ge 3$ and $G$ has a maximum matching $M$
containing an edge $u_1x$, $x\in V(T_i)$. Then by
Lemma~\ref{lem-2-3}, there exists a $G_1=U(T_1, \cdots, T_{c_i,
d_i}^*, \cdots, T_g)$ such that $W(G)\ge W(G_1)$ with equality if
and only if $T_i=T_{c_i, d_i}^*$, where $c_i+1$ is the matching
number of $T_i$ and $2c_i+d_i+1=|V(T_i)|$. Moreover, the matching
numbers of $G$ and $G_1$ are equal.

 {\bf Case 2:}  $|V(T_i)|\ge 3$ and  any maximum matching of $G$ does not contain
$u_ix$, $x\in V(T_i)$. Let   the matching number  of $T_i-u_i$ of
order $n_i-1$ be $a_i$. Then by Lemma~\ref{lem2-5}, there exists a
$G_2=U(T_1, \cdots, T_{c_i,d_i}^*, \cdots, T_g)$ such that $W(G)\ge
W(G_2)$ with equality if and only if $G=G_2$, where
$2c_i+d_i+1=n_i$. Moreover, the matching numbers of $G$ and $G_2$
are equal.

Hence there exists a $G_3=U(T_{c_1,d_1}^*, \cdots, T_{c_g,d_g}^*)$
such that $W(G)\ge W(G_3)$, and the matching numbers of $G$ and
$G_3$ are equal.
By the repeated use of Lemma~\ref{lem2-6}  and
Corollaries~\ref{corollary2-8} and \ref{corollary2-10}, it is easy
to see that the assertion holds.
\end{proof}

\section{Wiener index of unicyclic graphs with  girth and the matching number}
In this section, we  give a sharp lower bound for the Wiener index
of unicyclic graphs of order $n$ with girth $g$ and the matching
number $\beta\ge\frac{3g}{2}$ and characterize all extremal graphs
which attain the lower bound.  But we need some lemmas and notations

\begin{lemma}\label{lem3-2}
Let $G_1$ and $G_2$ be two simple connected graphs. Let $G$ be the
graph obtained from $G_1$ and $G_2$ by identifying a vertex $x$ of
$G_1$ and a vertex $y$ of $G_2$. Then
\begin{equation}\label{lem3-1-eq}
W(G)=W(G_1)+W(G_2)+d_{G_1}(x)(|V(G_2)|-1)+d_{G_2}(y)(|V(G_1)|-1).
\end{equation}
\end{lemma}

\begin{proof}
By the definition, we have
\begin{eqnarray*}
W(G)&=& \sum_{\{u, v\}\subseteq V(G_1)}d_G(u, v)+\sum_{\{u,
v\}\subseteq V(G_2)}d_G(u, v)+\sum_{u\in V(G_1)\backslash
\{x\}}\sum_{v\in V(G_2)\backslash\{y\}}d_G(u, v)\\
&=& W(G_1)+W(G_2)+\sum_{u\in V(G_1)\backslash \{x\}}\sum_{v\in
V(G_2)\backslash\{y\}}(d_{G_1}(u, x)+d_{G_2}(y, v))\\
&=& W(G_1)+W(G_2)+d_{G_1}(x)(|V(G_2)|-1)+d_{G_2}(y)(|V(G_1)|-1).
\end{eqnarray*}
Hence we finish the proof.
\end{proof}

Assume that $n\ge2\beta\ge 3g\ge 9$.  If $g$ is odd, let $G_{(n, g,
\beta)}^*$ be the unicyclic graph of order $n$ obtained by
identifying a vertex of a cycle $C_g$ of odd order $g$ and the
rooted vertex with degree $n-\beta-\frac{g-1}{2}$ of
$T_{\beta-\frac{g+1}{2}, n-2\beta+1}^*$ of order $n-g+1$. If $g$ is
even, let $G_{(n, g, \beta)}^*$  be a unicyclic graph of order $n$
obtained by  identifying  vertex  $u_1$ of a cycle $C_g=u_1u_2\cdots
u_g$ of even order $g$ and the rooted vertex with degree
$n-\beta-\frac{g}{2}$ of $T_{\beta-\frac{g}{2}-1, n-2\beta+1}^*$ of
order $n-g$, and adding a pendent edge $u_2v$ at vertex $u_2$. In
other words, \begin{eqnarray*}
G_{(n,g,\beta)}^*&=&U(T_{\beta-\frac{g+1}{2}, n-2\beta+1}^*,
T_{0,0}^*, \cdots, T_{0, 0}^*)\ \ \mbox{for }\ \ g\ \ \mbox{is \ \
odd},\\
G_{(n,g,\beta)}^*&=&U(T_{\beta-\frac{g}{2}-1, n-2\beta+1}^*,
T_{0,1}^*, \cdots, T_{0, 0}^*)\ \ \mbox{for }\ \ g\ \ \mbox{is \ \
even}.
\end{eqnarray*}
 Then $G_{(n, g, \beta)}^*$ is a
unicyclic graph of order $n$ with girth $g$ and the matching number
$\beta$. Moreover
\begin{corollary}\label{corollary3-2} (1). If $g$ is odd, then
\begin{equation}\label{corollary3-2-eq}
W(G_{(n,g,\beta)}^*)=n^2+\left(\beta-\frac{3g+1}{2}+\lfloor\frac{g^2}{4}\rfloor\right)n
+\left(1-\frac{g}{2}\right)\lfloor\frac{g^2}{4}\rfloor+g^2+\left(-2\beta+1\right)g-2\beta+1.
\end{equation}
(2). If $g$ is even, then
\begin{equation}\label{corollary3-2-eq-2}
W(G_{(n,g,\beta)}^*)=n^2+\left(\beta-\frac{3g}{2}-1+\lfloor\frac{g^2}{4}\rfloor\right)n
-\frac{g}{2}\lfloor\frac{g^2}{4}\rfloor+\frac{3g}{2}-3\beta+2.
\end{equation}
\end{corollary}
\begin{proof}
It follows from Lemma~\ref{lem3-2} and some calculation.
\end{proof}
\begin{lemma}\label{lemmma3-1}
Let $G=U(T_{a_1, b_1}^*, T_{0, b_2}^*, \cdots, T_{0, b_g}^*)$ be a
unicyclic graph of order $n$ with  odd girth  $ g$ and the matching
number $\beta$, where $b_i\le 1$ for $i=2, \cdots, g$. If $\beta\ge
\frac{3g}{2}$, then
 $$W(G)\ge W(G_{(n,g, \beta)}^*)=W(U(T_{\beta-\frac{g+1}{2}, n-2\beta+1}^*,
T_{0,0}^*, \cdots, T_{0, 0}^*))$$ with equality if and only if
$G=U(T_{\beta-\frac{g+1}{2}, n-2\beta+1}^*, T_{0,0}^*, \cdots, T_{0,
0}^*)$.
\end{lemma}
\begin{proof}   Let $t=\sum_{j=2}^gb_j$.  We consider the following two cases.

{\bf Case 1:}  $t=0$.  If $b_1>0$, then $a_1+\frac{g+1}{2}=\beta$,
so $a_1=\beta-\frac{g+1}{2}$ and $b_1=n-2\beta+1$ Hence $G$ must be
$U(T_{\beta-\frac{g+1}{2}, n-2\beta+1}^*, T_{0,0}^*, \cdots, T_{0,
0}^*)$ and the assertion holds. If $b_1=0$, then
$a_1=\beta-\frac{g-1}{2}$ and $n=\beta+g$. Hence
$a_1-2=\beta-\frac{g+1}{2}$ and $n-2\beta+1=2$. Further  we have
$$W(G)=W(U(T_{a_1,0}^*, T_{0, 0}^*,\cdots,
T_{0,0}^*))>W(U(T_{a_1-1,2}^*,T_{0, 0}^*,\cdots, T_{0,0}^*)).$$
Hence the assertion holds.

{\bf Case 2:} $t\ge 1$.  Suppose that the only cycle $C_g=u_1 \cdots
u_g$ and $T_{0,b_{i_1}}^*, \cdots, T_{0,b_{i_t}}^* $ consist of an
edge $u_{i_1}v_{i_1}$, $\cdots,$ $u_{i_t}v_{i_t}$, respectively,
where $2\le i_1<\cdots< i_t\le g$.  Let $V_1=\{v_{i1}, \cdots,
v_{it}\}$ and $V_2=V\setminus V_1$.  Then  $b_j=0$ for $j\neq 1,
i_1, \cdots, i_t$. Clearly, $\beta-g\le a_1\le \beta-\frac{g+1}{2}$.
Then $s\equiv \beta-\frac{g+1}{2}-a_1\ge 0$ and
$$r\equiv t-2s=(n-2a_1-b_1-g)-2(\beta-\frac{g+1}{2}-a_1)=n-2\beta-b_1+1\ge 0. $$
 Then $U(T_{a_1-1,2}^*,T_{0, 0}^*,\cdots, T_{0,0}^*)$ may be
 obtained from $G$ by deleting $u_{i_j}v_{i_j}$ for $j=1, \cdots, t$
 and adding $s$ paths of length 2, i.e.,  $u_1v_{i_1}v_{i_2}$, $\cdots, $ $
 u_ 1v_{i_{2s-1}}v_{i_{2s}}$ and $r$ edges $u_1v_{i_{2s+1}}$,
 $\cdots, $ $ u_1v_{i_t}$.  Therefore,
\begin{eqnarray*}
W(G)&=&\sum_{\{u, v\}\subseteq V_1}d_G(u, v)+\sum_{u\in V_1, v\in
V_2}d_G(u, v)+\sum_{\{u, v\}\subseteq V_2}d_G(u, v)\\
&\ge&
\frac{3t(t-1)}{2}+t(g+\lfloor\frac{g^2}{4}\rfloor+7a_1+3b_1)+\sum_{\{u,
v\}\subseteq V_2}d_G(u, v).
\end{eqnarray*}
On the other hand,
\begin{eqnarray*}
&&W(G_1)=\sum_{\{u, v\}\subseteq V_1}d_{G_1}(u, v)+\sum_{u\in V_1,
v\in
V_2}d_{G_1}(u, v)+\sum_{\{u, v\}\subseteq V_2}d_{G_1}(u, v)\\
&=& [6s^2+(5r-5)s]+[(3s+r)g+t\lfloor\frac{g^2}{4}\rfloor
+(12s+5r)a_1+(5s+2r)b_1]+\sum_{\{u, v\}\subseteq V_2}d_{G_1}(u, v).
\end{eqnarray*}
Hence
\begin{eqnarray*}
W(G)-W(G_1)&=&(r+2)s+\frac{3r(r-1)}{2}+(s+r)b_1+(2s+2r)a_1-g s \\
&>&0,\end{eqnarray*} since $a_1\ge
\beta-g\ge\frac{3}{2}g-g=\frac{g}{2}$.
\end{proof}

\begin{theorem}\label{theorem3-3}
Let $G$ be a unicyclic graph of order $n$ with odd girth $g$ and the
matching number $\beta$.  If $\beta\ge \frac{3g}{2},$ then
$$W(G)\ge n^2+\left(\beta-\frac{3g+1}{2}+\lfloor\frac{g^2}{4}\rfloor\right)n
+\left(1-\frac{g}{2}\right)\lfloor\frac{g^2}{4}\rfloor+g^2+\left(-2\beta+1\right)g-2\beta+1$$
with equality if and only if $G$ is $G_{(n, g, \beta)}^*$.
\end{theorem}
\begin{proof} It follow from Theorem~\ref{th2-8}, Lemmas~\ref{lemmma3-1} and
Corollary~\ref{corollary3-2} that the assertion holds.
\end{proof}

\begin{lemma}\label{lemma3-5}
Let $G=U(T_{a_1, b_1}^*, T_{0, b_2}^*, \cdots, T_{0, b_g}^*)$ be a
unicyclic graph of order $n$ with even girth $g$ and the matching
number $\beta\ge\frac{3g}{2}$. If $a_1\le \beta-\frac{g}{2}-1$ and
$b_j\le 1$ for $j=2, \cdots, g$,  then
 $$W(G)\ge W(G_{(n, g, \beta)}^*)\equiv W(U(T_{\beta-\frac{g}{2}-1, n-2\beta+1}^*,
T_{0,1}^*, \cdots, T_{0, 0}^*)).
 $$
 with equality if and only if $G= G_{(n, g, \beta)}^*$.
\end{lemma}
\begin{proof}
Since $a_1\le \beta-\frac{g}{2}-1$, we have
$t\equiv\sum_{j=2}^tb_j\ge 1$. Suppose that the only cycle
$C_g=u_1\cdots u_g$ and  $T_{0,b_{i_j}}^* $
 is an edge $u_{i_j}v_{i_j}$, $j=1,\cdots,t$,
  where $1\le i_1<\cdots<i_t\le g$.
 Let $V_1=\{v_{i1}, \cdots,
v_{it}\}$ and $V_2=V\setminus V_1$.  Then
$s\equiv\beta-\frac{g}{2}-1-a_1\ge 0$ and
$$r\equiv t-2s-1=(n-2a_1-b_1-g)-2(\beta-\frac{g}{2}-1-a_1)-1=n-2\beta-b_1+1\ge 0.$$
  Hence $G_{(n, g, \beta)}^*$ may be obtained from $G$ by deleting
$\{v_{i_1}, \cdots, v_{i_t}\}$ and adding  $s$ paths
$u_1v_{i_{2l-1}}v_{i_{2l}}$ for $l=1, \cdots, s$ and adding edges
$u_1v_{i_l}$ for $l=2s+1,\cdots, t-1$ and $u_2v_{i_t}$.  Further,
\begin{eqnarray*}
W(G)&=& \sum_{\{u, v\}\subseteq V_1}d_G(u, v)+\sum_{u\in V_1, v\in
V_2}d_G(u, v)+\sum_{\{u, v\}\subseteq V_2}d_G(u, v)\\
&\ge &
  3\frac{t(t-1)}{2}+t\left(
g+\lfloor\frac{g^2}{4}\rfloor +7a_1+3b_1\right)+\sum_{\{u,
v\}\subseteq
V_2}d_G(u, v)\\
&=& 6s^2+(6r+3)s+\frac{3r(r+1)}{2}+t\left(
g+\lfloor\frac{g^2}{4}\rfloor +7a_1+3b_1\right)+\sum_{\{u,
v\}\subseteq V_2}d_G(u, v).
\end{eqnarray*}

On the other hand,
\begin{eqnarray*}
\sum_{\{u, v\}\subseteq V_1}d_{G_{(n,g,\beta)}^*}(u, v)
=6s^2+(5r+2)s+r(r+1)
\end{eqnarray*}
and
\begin{eqnarray*}
\sum_{u\in V_1, v\in V_2}d_{G_{(n,g,\beta)}^*}(u, v)=
(3s+r+1)g+t\lfloor\frac{g^2}{4}\rfloor+(12s+5r+7)a_1+(5s+2r+3)b_1.
\end{eqnarray*}
Hence
$$W(G)-W(G_{(n,g,\beta)}^*)=(r+1)s+\frac{r(r+1)}{2}+(s+r)b_1+2(s+r)a_1-sg\ge 0$$
with equality if and only if $r=s=0$, since $a_1\ge \beta-g\ge
\frac{g}{2}$. Hence the assertion holds.
\end{proof}

\begin{lemma}\label{lemma3-6}
Let $G=U(T_{a_1,b_1}^*, T_{0, b_2}^*, \cdots, T_{0,b_g}^*)$ be a
unicyclic graph of order $n$ with even girth $g$ and the matching
number $\beta\ge\frac{3g}{2}$.
 If $a_1= \beta-\frac{g}{2}$ and $b_j\le 1$ for $j=2, \cdots, g$, then
 $$W(G)\ge W(U(T_{a_1, b_1+t}^*, T_{0,0}^*, \cdots, T_{0, 0}^*))
 $$
 with equality if and only if $G$ is $U(T_{a_1, b_1+t}^*, T_{0,0}^*, \cdots, T_{0,
 0}^*)$, where $t=\sum_{i=2}^gb_i$.
\end{lemma}
\begin{proof}  If $t=\sum_{i=2}^gb_i=0$, then the assertion holds. Suppose that $t\ge 1$.
Then the matching number of $ U(T_{a_1, b_1+t}^*, T_{0,0}^*, \cdots,
T_{0, 0}^*)$ is $\beta$. Moreover, $ G_1\equiv U(T_{a_1,b_1+t}^*,$ $
T_{0,0}^*, $ $ \cdots, $ $ T_{0, 0}^*)$  may  be obtained from $G$
by deleting vertices $v_{i_1}, \cdots, v_{i_t}$ and adding edges
$u_1v_{i_1},\cdots, u_1v_{i_t}$. Then it is easy to see that $W(G)>
W(G_1).
 $ Therefore the proof is finished.
\end{proof}

\begin{theorem}\label{theorem3-7}
Let $G$ be a unicyclic graph of order $n$ with even girth $g$ and
the matching number $\beta$.  If $\beta\ge \frac{3g}{2},$ then
$$W(G)\ge W(G_{(n, g, \beta)}^*)=n^2+\left(\beta-\frac{3g}{2}-1+\lfloor\frac{g^2}{4}\rfloor\right)n
-\frac{g}{2}\lfloor\frac{g^2}{4}\rfloor+\frac{3g}{2}-3\beta+2$$ with
equality if and only if $G$ is $G_{(n, g, \beta)}^*$.
\end{theorem}
\begin{proof}
By Theorem~\ref{th2-8}, there exists a unicyclic graph  $G_1$ of
order $n$ with even girth $g$ and the matching number $\beta$ such
that $G_1=U(T_{a_1, b_1}^*, T_{0, b_2}^*,\cdots, T_{0,b_g}^*)$ of
order $n$ with girth $g$ and the matching number $\beta$ such that
$W(G)\ge W(G_1)$, where $b_i\le 1$ for $i=2, \cdots, g$. If $a_1\le
\beta-\frac{g}{2}-1$, then by Lemma~\ref{lemma3-5}, $W(G_1)\ge
W(U_{(n,g,\beta)}^*)$ with  equality if and only if $G_1$ is
$U_{(n,g,\beta)}^*$.  If $a_1=\beta-\frac{g}{2},$  then by
Lemma~\ref{lemma3-6}, we have $W(G_1)\ge W(U(T_{\beta-\frac{g}{2},
n-2\beta}^*, T_{0,0}^*, \cdots, T_{0,
 0}^*)$.  Further, it is easy to see that
  $$W(U(T_{\beta-\frac{g}{2},
n-2\beta}^*, T_{0,0}^*, \cdots, T_{0,
 0}^*)> W(U(T_{\beta-\frac{g}{2}-1,
n-2\beta+1}^*, T_{0,1}^*, \cdots, T_{0,
 0}^*)).$$
Therefore  by Corollary~\ref{corollary3-2}, the assertion holds.
\end{proof}

Combining Theorems~\ref{theorem3-3} and \ref{theorem3-7}, we obtain
the main result in this paper
\begin{theorem}
Let $G$ be a unicyclic graph of order $n$ with  girth $g$ and the
matching number $\beta$.  If $\beta\ge \frac{3g}{2},$ then
$$W(G)\ge
W(G_{(n, g, \beta)}^*)$$ with equality if and only if $G$ is $G_{(n,
g, \beta)}^*$.
\end{theorem}



\end{document}